\documentclass[11pt]{amsart}%
\usepackage{amscd,amsmath,latexsym,amsthm,amsfonts,amssymb,graphicx,color,geometry}
\usepackage{enumerate}
\usepackage[pdftex,plainpages=false,pdfpagelabels,
colorlinks=true,citecolor=blue,linkcolor=black,urlcolor=black,
filecolor=black,bookmarksopen=true,unicode=true]{hyperref}
\usepackage[norefs,nocites]{refcheck}
\usepackage{amsmath}
\usepackage{amsfonts}
\usepackage{amssymb}
\usepackage[T1]{fontenc}
\usepackage{graphicx}
\usepackage{xcolor}
\usepackage{colortbl}
\usepackage{float}
\usepackage{framed}
\usepackage{pgf,tikz,pgfplots}
\usepackage{tkz-base}
\usepackage{multicol}
\usepackage{tabularx}
\usepackage{tkz-fct}%
\providecommand{\U}[1]{\protect\rule{.1in}{.1in}}
\tikzset{>=Triangle}
\newtheorem{theorem}{Theorem}[section]

\newtheorem{corollary}[theorem]{Corollary}

\newtheorem{remark}[theorem]{Remark}

\newtheorem{lemma}[theorem]{Lemma}

\newtheorem{definition}[theorem]{Definition}
\numberwithin{equation}{section}
\pgfplotsset{compat=1.17}
\begin{document}
\title[General criteria for a stronger notion of lineability]{General criteria for a stronger notion of lineability}
\author[V.V. F\'{a}varo]{Vin\'icius V. F\'{a}varo}
\address{Faculdade de Matem\'{a}tica \\
Universidade Federal de Uberl\^{a}ndia \\
38400-902 - U\-ber\-l\^{a}n\-dia, Brazil.}
\email{vvfavaro@ufu.br}
\author[D.\ Pellegrino]{Daniel Pellegrino}
\address{Departamento de Matem\'{a}tica \\
Universidade Federal da Para\'{\i}ba \\
58.051-900 - Jo\~{a}o Pessoa, Brazil.}
\email{daniel.pellegrino@academico.ufpb.br}
\author[A. Raposo Jr.]{Anselmo Raposo Jr.}
\address{Departamento de Matem\'{a}tica \\
Universidade Federal do Maranh\~{a}o \\
65085-580 - S\~{a}o Lu\'{\i}s, Brazil.}
\email{anselmo.junior@ufma.br}
\author[G. Ribeiro]{Geivison Ribeiro}
\address{Departamento de Matem\'{a}tica \\
Universidade Federal da Para\'{\i}ba \\
58.051-900 - Jo\~{a}o Pessoa, Brazil.}
\email{geivison.ribeiro@academico.ufpb.br}
\thanks{V.V. F\'{a}varo is supported by FAPEMIG Grant PPM-00217-18 and partially
supported by FAPEMIG Grant RED-00133-21}
\thanks{D. Pellegrino is supported by CNPq Grant 307327/2017-5 and Grant 2019/0014
Para\'{\i}ba State Research Foundation (FAPESQ) }
\subjclass[2020]{15A03, 47B37, 47L05, 46B86}
\keywords{Lineability; spaceability}

\begin{abstract}
A subset $A$ of a vector space $X$ is called $\alpha$-lineable whenever $A$
contains, except for the null vector, a subspace of dimension $\alpha$. If $X$
has a topology, then $A$ is $\alpha$-spaceable if such subspace can be chosen
to be closed. The vast existing literature on these topics has shown that
positive results for lineability and spaceability are quite common. Recently,
the stricter notions of $\left(  \alpha,\beta\right)  $%
-lineability/spaceability were introduced as an attempt to shed light to more
subtle issues. In this paper, among other results, we prove some general
criteria for the notion of $\left(  \alpha,\beta\right)  $-spaceability and,
as applications, we extend recent results of different authors.

\end{abstract}
\maketitle

\section{Introduction}

The notions of lineability and spaceability were introduced in the seminal
paper \cite{AGSS} by Aron, Gurariy and Seoane-Sep\'{u}lveda and its essence is
to investigate linear structures within exotic settings. More precisely,
$A\subset V$ is $\alpha$-lineable in a vector space $V$, if $A\cup\left\{
0\right\}  $ contains an $\alpha$-dimensional subspace of $V$. In addition, if
$V$ is endowed with a topology, we say that $A$ is $\alpha$-spaceable in $V$
if there is an $\alpha$-dimensional closed subspace of $V$ contained in
$A\cup\left\{  0\right\}  $.

It turns out that the vast literature related to this subject has shown that
positive results of lineability and spaceability are rather common; we refer
the reader to \cite{aron}. Recently, in \cite{FPT} the notions of $\left(
\alpha,\beta\right)  $-lineability and\textit{ }$\left(  \alpha,\beta\right)
$\textit{-}spaceability were introduced as an attempt to investigate how far
positive results of lineability and spaceability remain valid under stricter assumptions.

From now on $\mathbb{N}$ denotes the set of all positive integers and
$\mathbb{K}$ represents the real scalar field $\mathbb{R}$ or the complex
scalar field $\mathbb{C}$. All vector spaces are considered over $\mathbb{K}$.

Let $V$ be a vector space and let $A$ be a non-void subset of $V$.

\begin{itemize}
\item If $\alpha,\beta$ are cardinal numbers, $\alpha\leq\beta$, and $V$ is a
vector space, then $A$ is called $\left(  \alpha,\beta\right)  $-lineable if
it is $\alpha$-lineable and for each $\alpha$-dimensional subspace $W_{\alpha
}\subset A\cup\left\{  0\right\}  $ there is a $\beta$-dimensional subspace
$W_{\beta}$ such that
\begin{equation}
W_{\alpha}\subset W_{\beta}\subset A\cup\left\{  0\right\}  \text{.}
\label{1.1}%
\end{equation}

\item When $V$ is endowed with a topology and the subspace $W_{\beta}$
satisfying (\ref{1.1}) can always be chosen closed, we say that $A$ is
$\left(  \alpha,\beta\right)  $-spaceable.
\end{itemize}

Note that, in particular, $\left(  \alpha,\beta\right)  $%
-lineability/spaceability implies $\beta$-lineability/spaceability. On the
other hand, while $\left(  \alpha,\alpha\right)  $-lineability and $\alpha
$-lineability are the same concept, $\left(  \alpha,\alpha\right)
$-spaceability does not coincide with $\alpha$-spaceability.

In the same direction, in \cite{Pellegrino} the notion of pointwise $\beta
$-lineability/spaceability was introduced as follows:

\begin{itemize}
\item A subset $A$ of a vector space $V$ is called pointwise $\beta$-lineable
(spaceable) if, for each $x\in A$, there is a (closed) $\beta$-dimensional
subspace $W_{x}$ such that
\[
x\in W_{x}\subset A\cup\left\{  0\right\}  \text{.}%
\]

\end{itemize}

It is plain that the notion of pointwise $\beta$-lineability (spaceability)
implies $\left(  1,\beta\right)  $-lineability (spaceability). However, in
general, the converse is not true (see \cite[Example 2.2]{Pellegrino}).

Up to now there are only few results on the $\left(  \alpha,\beta\right)  $
notions of lineability and\textit{ }spaceability (see \cite{Diogo/Anselmo,
Diogo, Pilar, FPT, Pellegrino}). Most of the results obtained thus far prove
$\left(  1,\mathfrak{c}\right)  $-lineability and $\left(  1,\mathfrak{c}%
\right)  $-spaceability of certain sets, and no general technique relating
lineability/spaceability and $\left(  \alpha,\beta\right)  $%
-lineability/spaceability seems to be known. The main goal of this paper is to
shed light to the interplay between the notions of lineability/spaceability
and $\left(  \alpha,\beta\right)  $-lineability/spaceability, providing
general techniques which are applied in several different settings.

The letters $\alpha,\beta,\lambda$ will always represent cardinal numbers,
$\operatorname*{card}\left(  A\right)  $ denotes the cardinality of the set
$A$, $\aleph_{0}:=\operatorname*{card}\left(  \mathbb{N}\right)  $ and
$\mathfrak{c}:=\operatorname*{card}\left(  \mathbb{R}\right)  $.

The paper is organized as follows. In Section \ref{Sec2} we prove a general
result on $\left(  \alpha,\beta\right)  $-spaceability and, as a consequence,
we conclude that the sets $L_{p}\left[  0,1\right]  \setminus\bigcup
\nolimits_{q\in\left(  p,\infty\right)  }L_{q}\left[  0,1\right]  $ and
$\mathcal{ND}\left[  0,1\right]  $ of all continuous nowhere differentiable
functions from $\left[  0,1\right]  $ to $\mathbb{R}$ are not $\left(
\alpha,\beta\right)  $-spaceable whenever $\alpha\geq\aleph_{0}$. These
results complement previous results of Botelho et al. \cite{BFPS} and Gurariy
\cite{FGK}. In Section \ref{Sec3}, we develop a general technique that allows
us to conclude, for instance, that $\ell_{\infty}\setminus c_{0}$ and
$\ell_{\infty}\setminus c$ are $\left(  \alpha,\mathfrak{c}\right)
$-spaceable if, and only if, $\alpha<\aleph_{0}$. Sections \ref{Sec4} and
\ref{Sec5} are devoted to the investigation of dense lineability. We prove
general results which, in particular, yields that $L_{p}[0,1]\setminus
\bigcup\nolimits_{q\in\left(  p,\infty\right)  }L_{q}[0,1]$ is pointwise
maximal dense-lineable and $\left(  n,\aleph_{0}\right)  $-dense lineable, for
each $n\in\mathbb{N}$.

\section{$\left(  \alpha,\beta\right)  $-spaceable sets: negative
results\label{Sec2}}

We start off recalling that (see \cite[Definition 2.1]{Aron2009}) if $A$, $B$
are subsets of a vector space $V$, then $A$ is called stronger than $B$
whenever $A+B\subset A$. We also recall that an $F$-space is a topological
vector space whose topology can be defined by a complete translation-invariant metric.

The main result of this section reads as follows:

\begin{theorem}
\label{Teo2.1}Let $\alpha\geq\aleph_{0}$ and $V$ be an $F$-space. Let $A$, $B$
be subsets of $V$ such that $A$ is $\alpha$-lineable and $B$ is $1$-lineable.
If $A\cap B=\varnothing$ and $A$ is stronger than $B$, then $A$ is not
$\left(  \alpha,\beta\right)  $-spaceable, regardless of the cardinal number
$\beta$.
\end{theorem}

\begin{proof}
Consider a complete translation-invariant metric $d$ on $V$ whose topology is
the one of $V$. Since $B$ is $1$-lineable, let $v\in V\setminus\left\{
0\right\}  $ be such that
\[
\mathbb{K}v\setminus\left\{  0\right\}  \subset B\text{.}%
\]
Since $A$ is $\alpha$-lineable, let $\Gamma$ be a set with cardinality
$\alpha$ and $\left\{  v_{a}:a\in\Gamma\right\}  \subset V$ be a set of
linearly independent vectors such that
\[
E\setminus\left\{  0\right\}  \subset A\text{,}%
\]
where $E=\operatorname*{span}\left\{  v_{a}:a\in\Gamma\right\}  $. Let
$\left\{  a_{m}:m\in\mathbb{N}\right\}  \subset\Gamma$ be an infinite
countable set. Given $n\in\mathbb{N}$, the continuity of scalar multiplication
on $s\mapsto sv_{a_{n}}$ yields that there exists $\delta_{n}>0$ such that%
\[
d\left(  sv_{a_{n}},0v_{a_{n}}\right)  <n^{-1}\text{, for all }s\in\left(
0,\delta_{n}\right)  \text{. }%
\]
In particular, for $s=2^{-1}\delta_{n}$, we have
\[
d\left(  2^{-1}\delta_{n}v_{a_{n}},0\right)  <n^{-1}\text{.}%
\]
Given $b\in\Gamma$, define $u_{b}:=t_{b}v_{b}+v$, where%
\[
t_{b}=\left\{
\begin{array}
[c]{ll}%
2^{-1}\delta_{n}\text{,} & \text{if }b=a_{n}\in\left\{  a_{m}\right\}
_{m=1}^{\infty}\text{ for some }n\text{,}\vspace{0.2cm}\\
1\text{,} & \text{otherwise.}%
\end{array}
\right.
\]
Since $A+B\subset A$, it follows that $u_{b}\in A$, for each $b\in\Gamma$.
Since $A\cap B=\varnothing$, it is clear that $\left\{  u_{b}:b\in
\Gamma\right\}  $ is linearly independent. Defining
\[
M:=\operatorname*{span}\left\{  u_{b}:b\in\Gamma\right\}  \text{,}%
\]
let us prove that $M\setminus\left\{  0\right\}  \subset A$. For $f\in
M\setminus\left\{  0\right\}  $, there are $N\in\mathbb{N}$ and $\left(
c_{1},\ldots,c_{N}\right)  \neq(0,\ldots,0)$ in $\mathbb{K}^{N}$\ such that
\[
f=%
%TCIMACRO{\tsum \limits_{j=1}^{N}}%
%BeginExpansion
{\textstyle\sum\limits_{j=1}^{N}}
%EndExpansion
c_{j}t_{b_{j}}v_{_{bj}}+%
%TCIMACRO{\tsum \limits_{j=1}^{N}}%
%BeginExpansion
{\textstyle\sum\limits_{j=1}^{N}}
%EndExpansion
c_{j}v\text{.}%
\]
Since $\left\{  v_{b}:b\in\Gamma\right\}  $ is linearly independent and
$\left(  c_{1},\ldots,c_{N}\right)  \neq(0,\ldots,0)$, we have
\[%
%TCIMACRO{\tsum \limits_{j=1}^{N}}%
%BeginExpansion
{\textstyle\sum\limits_{j=1}^{N}}
%EndExpansion
c_{j}t_{b_{j}}v_{_{bj}}\in E\setminus\left\{  0\right\}  \subset A\text{,}%
\]
and since
\[%
%TCIMACRO{\tsum \limits_{j=1}^{N}}%
%BeginExpansion
{\textstyle\sum\limits_{j=1}^{N}}
%EndExpansion
c_{j}v\in\mathbb{K}v\subset B\cup\left\{  0\right\}  \text{,}%
\]
we get
\[
f\in A+\left(  B\cup\left\{  0\right\}  \right)  \subset A\text{.}%
\]
Hence $M\setminus\left\{  0\right\}  \subset A$. Given $\varepsilon>0$, let
$n\in\mathbb{N}$ be such that $n^{-1}<$ $\varepsilon$. Since $d$ is a
translation-invariant metric, we conclude that
\[
d\left(  u_{a_{n}},v\right)  =d\left(  2^{-1}\delta_{n}v_{a_{n}},0\right)
<n^{-1}<\varepsilon\text{.}%
\]
This implies that $v\in\overline{M}$. Since $v\notin A$, it follows that
\[
\overline{M}\not \subset A\cup\left\{  0\right\}
\]
and this means that $A$ is not $\left(  \alpha,\beta\right)  $-spaceable,
regardless of the $\beta\geq$ $\alpha$.
\end{proof}

From \cite{FGK} we know that the set $\mathcal{ND}\left[  0,1\right]  $ of
continuous nowhere differentiable functions $f\colon\left[  0,1\right]
\rightarrow\mathbb{R}$ is $\mathfrak{c}$-spaceable in $C\left[  0,1\right]  $.
The following consequence of the previous result shows that it is not $\left(
\alpha,\mathfrak{c}\right)  $-spaceable regardless of the $\alpha\geq
\aleph_{0}$:

\begin{corollary}
Let $\alpha\geq\aleph_{0}$ and $\beta$ be a cardinal number. The set
$\mathcal{ND}\left[  0,1\right]  $ is not $\left(  \alpha,\beta\right)  $-spaceable.
\end{corollary}

\begin{proof}
Consider $A=\mathcal{ND}\left[  0,1\right]  $ and $B=\left\{  f\in C\left[
0,1\right]  :f\text{ is differentiable}\right\}  $. Note that
\[
A\cap B=\varnothing\text{ and }A+B\subset A\text{.}%
\]
If $\aleph_{0}\leq\alpha\leq\mathfrak{c,}$ since $A$ is $\alpha$-lineable, the
result follows from Theorem \ref{Teo2.1}. If $\alpha>\mathfrak{c}$ the result
is immediate.
\end{proof}

\begin{corollary}
\label{Cor2.3}Let $\alpha\geq\aleph_{0}$ and $\beta$ be a cardinal number. Let
$X$ be a Banach space or $p$-Banach space $\left(  p>0\right)  $ and $Y$ be a
non-trivial subspace of $X$. If $X\setminus Y$ is $\alpha$-lineable then
$X\setminus Y$ is not $\left(  \alpha,\beta\right)  $-spaceable.
\end{corollary}

\begin{proof}
Considering $A=X\setminus Y$ and $B=Y$, we have $A+B\subset A$ and $A\cap
B=\varnothing$ and the result follows by Theorem \ref{Teo2.1}.
\end{proof}

It is well known (see \cite{BFPS}) that $L_{p}\left[  0,1\right]
\setminus\bigcup\nolimits_{q\in\left(  p,\infty\right)  }L_{q}\left[
0,1\right]  $, for $p>0$, is $\mathfrak{c}$-spaceable in $L_{p}\left[
0,1\right]  $. Considering $X=L_{p}[0,1]$ and $Y=%
%TCIMACRO{\tbigcup \nolimits_{q\in(p,\infty)}}%
%BeginExpansion
{\textstyle\bigcup\nolimits_{q\in(p,\infty)}}
%EndExpansion
L_{q}[0,1]$ in Corollary \ref{Cor2.3}, since $X\setminus Y$ is $\alpha
$-lineable for all $\aleph_{0}\leq\alpha\leq\mathfrak{c}$ and $\dim
X=\mathfrak{c}$, we have:

\begin{corollary}
Let $\alpha\geq\aleph_{0}$ and $\beta$ be a cardinal number. For $p>0$, the
set $L_{p}[0,1]\setminus%
%TCIMACRO{\tbigcup \nolimits_{q\in(p,\infty)}}%
%BeginExpansion
{\textstyle\bigcup\nolimits_{q\in(p,\infty)}}
%EndExpansion
L_{q}[0,1]$ is not $\left(  \alpha,\beta\right)  $-spaceable.
\end{corollary}

\begin{remark}
Observing the proof of Theorem \ref{Teo2.1}, we are led to believe that a
(natural) variant of the definition of $\left(  \alpha,\beta\right)
$-spaceability, demanding $W_{\alpha}$ to be closed, would provide a
completely different kind of results. We think that this is worth of further
investigation in the future.
\end{remark}

\section{$\left(  \alpha,\beta\right)  $-spaceable sets: positive
results\label{Sec3}}

A classical result due to Wilansky and Kalton (see \cite[Theorem 2.2]{Kitson})
says that if $W$ is a closed subspace of a Fr\'{e}chet space $V$, then
$V\setminus W$ is spaceable\emph{ }if, and only if, $\dim V/W=\infty$:

\begin{theorem}
\emph{(}\cite[Theorem 2.2]{Kitson}\emph{)} If $W$ is a closed vector subspace
of a Fr\'{e}chet space $V$, then $V\setminus W$ is spaceable if, and only if,
$W$ has infinite codimension.
\end{theorem}

In this section, in some sense, we complement this result by showing that, if
$W$ is a closed subspace of a Banach space $V$ such that $W$ has a regular
basic sequence and $V\setminus W$ is $\aleph_{0}$-lineable, then

\begin{center}
$V\setminus W$ is $\left(  \alpha,\mathfrak{c}\right)  $-spaceable if, and
only if, $\alpha<\aleph_{0}$.
\end{center}

\noindent In particular, this assures that $\ell_{\infty}\setminus c_{0}$ and
$\ell_{\infty}\setminus c$ are $\left(  \alpha,\mathfrak{c}\right)
$-spaceable if, and only if, $\alpha<\aleph_{0}$.

We start off recalling some results concerning Schauder basis and basic sequences.

\begin{definition}
A sequence $\left(  e_{n}\right)  _{n=1}^{\infty}$ in a Banach space $V$ is a
basic sequence if it is a Schauder basis for $\overline{\operatorname*{span}%
\left\{  e_{n}:n\in\mathbb{N}\right\}  }$. Furthermore, as in \cite{Kalton}
and \cite{Shapiro}, a basic sequence is called regular if it is bounded away
from zero, that is, if it lies entirely outside some neighborhood of zero.
\end{definition}

The following lemmas are probably folklore, but we present their proofs for
the sake of completeness.

\begin{lemma}
\label{Lema3.3}Let $V$ be an infinite dimensional Banach space and let $W$ be
a finite-dimensional subspace of $V$. Then, given $\varepsilon\in\left(
0,1\right)  $, there exists $v\in V$ such that $\left\Vert v\right\Vert =1$
and
\[
\left\Vert w+\lambda v\right\Vert \geq\left(  1-\varepsilon\right)  \left\Vert
w\right\Vert
\]
for all $\left(  \lambda,w\right)  \in\mathbb{K}\times W$.
\end{lemma}

\begin{proof}
Let $\mathbb{S}_{W}$ be the unit sphere of $W$. Since $\mathbb{S}_{W}$ is
compact, we can find a finite set $\left\{  w_{1},\ldots,w_{r}\right\}  $ in
$\mathbb{S}_{W}$ such that $\left\{  B_{\varepsilon}\left(  w_{k}\right)
:k=1,\ldots,r\right\}  $, where $B_{\varepsilon}\left(  w_{k}\right)  $ is the
open ball of center $w_{k}$ and radius $\varepsilon$, covers $\mathbb{S}_{W}$.
Let $V^{\ast}$ be the topological dual of $V$ and let us consider $w_{1}%
^{\ast},\ldots,w_{r}^{\ast}\in V^{\ast}$ such that, for every $k=1,\ldots,r$,
$\left\Vert w_{k}^{\ast}\right\Vert =1$ and $w_{k}^{\ast}\left(  w_{k}\right)
=1$. Since $V$ is infinite dimensional, $\bigcap\limits_{k=1}^{r}\ker
w_{k}^{\ast}$ is non-trivial subspace of $V$.

Let us fix a unit vector $v\in\bigcap\limits_{k=1}^{r}\ker w_{k}^{\ast}$.
Given $w\in W\setminus\left\{  0\right\}  $, there is $k\in\left\{
1,\ldots,r\right\}  $ such that%
\[
\left\Vert \dfrac{w}{\left\Vert w\right\Vert }-w_{k}\right\Vert <\varepsilon
\text{.}%
\]
For any $\lambda\in\mathbb{K}$, we have%
\begin{align*}
\dfrac{\left\Vert w+\lambda v\right\Vert }{\left\Vert w\right\Vert }  &
\geq\left\Vert w_{k}+\dfrac{\lambda v}{\left\Vert w\right\Vert }\right\Vert
-\left\Vert \dfrac{w}{\left\Vert w\right\Vert }-w_{k}\right\Vert \\
&  >\left\Vert w_{k}+\dfrac{\lambda v}{\left\Vert w\right\Vert }\right\Vert
-\varepsilon\\
&  \geq\left\vert w_{k}^{\ast}\left(  w_{k}+\dfrac{\lambda v}{\left\Vert
w\right\Vert }\right)  \right\vert -\varepsilon\text{ }\\
&  =\left\vert w_{k}^{\ast}(w_{k})\right\vert -\varepsilon\\
&  =1-\varepsilon\text{.}%
\end{align*}
Hence%
\[
\left\Vert w+\lambda v\right\Vert \geq\left(  1-\varepsilon\right)  \left\Vert
w\right\Vert \text{,}%
\]
for all $w\in W\setminus\left\{  0\right\}  $ and all scalars $\lambda$. Since
the case $w=0$ is immediate, the proof is done.
\end{proof}

\begin{lemma}
\label{Lema3.4}Let $V$ be an infinite dimensional Banach space. If
$v_{1},\ldots,v_{n}\in V$ are linearly independent with $\left\Vert
v_{i}\right\Vert =1$, for each $i=1,\dots,n$, then $V$ contains a basic
sequence $\left(  u_{k}\right)  _{k=1}^{\infty}$ where $u_{k}=v_{k}$ for each
$k=1,\ldots,n$.
\end{lemma}

\begin{proof}
Let $W_{1}=\operatorname*{span}\left\{  v_{1},\ldots,v_{n}\right\}  $.
Considering the sequence $\left(  \varepsilon_{k}\right)  _{k=1}^{\infty}$
defined by%
\[
\varepsilon_{k}=\left(  \frac{10^{k}}{9}+%
%TCIMACRO{\tsum \limits_{j=0}^{k-1}}%
%BeginExpansion
{\textstyle\sum\limits_{j=0}^{k-1}}
%EndExpansion
10^{j}\right)  ^{-1}\text{,}%
\]
observe that%
\[
\left(
%TCIMACRO{\tprod \limits_{k=1}^{m}}%
%BeginExpansion
{\textstyle\prod\limits_{k=1}^{m}}
%EndExpansion
\left(  1-\varepsilon_{k}\right)  \right)  ^{-1}\leq%
%TCIMACRO{\tprod \limits_{k=1}^{\infty}}%
%BeginExpansion
{\textstyle\prod\limits_{k=1}^{\infty}}
%EndExpansion
\left(  1-\varepsilon_{k}\right)  ^{-1}=2
\]
for all $m$.

Define $u_{k}=v_{k}$ for each $k=1,\ldots,n$. By Lemma \ref{Lema3.3} there is
$u_{n+1}\in V$ such that $\left\Vert u_{n+1}\right\Vert =1$ and
\[
\left\Vert w+\lambda u_{n+1}\right\Vert \geq\left(  1-\varepsilon_{1}\right)
\left\Vert w\right\Vert
\]
for all $w\in W_{1}$ and $\lambda\in\mathbb{K}$, that is
\[
\left\Vert
%TCIMACRO{\tsum \limits_{k=1}^{n}}%
%BeginExpansion
{\textstyle\sum\limits_{k=1}^{n}}
%EndExpansion
\lambda_{k}u_{k}+\lambda u_{n+1}\right\Vert \geq\left(  1-\varepsilon
_{1}\right)  \left\Vert
%TCIMACRO{\tsum \limits_{k=1}^{n}}%
%BeginExpansion
{\textstyle\sum\limits_{k=1}^{n}}
%EndExpansion
\lambda_{k}u_{k}\right\Vert
\]
whenever $\lambda,\lambda_{1},\ldots,\lambda_{n}\in\mathbb{K}$. Defining
$W_{2}=\operatorname*{span}\left\{  u_{1},\ldots,u_{n},u_{n+1}\right\}  $, it
follows from Lemma \ref{Lema3.3} that there is $u_{n+2}\in V$ such that
$\left\Vert u_{n+2}\right\Vert =1$ and
\[
\left\Vert w+\lambda u_{n+2}\right\Vert \geq\left(  1-\varepsilon_{2}\right)
\left\Vert w\right\Vert
\]
for all $w\in W_{2}$ and $\lambda\in\mathbb{K}$, that is,
\[
\left\Vert
%TCIMACRO{\tsum \limits_{k=1}^{n+2}}%
%BeginExpansion
{\textstyle\sum\limits_{k=1}^{n+2}}
%EndExpansion
\lambda_{k}u_{k}\right\Vert \geq\left(  1-\varepsilon_{2}\right)  \left\Vert
%TCIMACRO{\tsum \limits_{k=1}^{n+1}}%
%BeginExpansion
{\textstyle\sum\limits_{k=1}^{n+1}}
%EndExpansion
\lambda_{k}u_{k}\right\Vert
\]
whenever $\lambda_{1},\ldots,\lambda_{n+2}\in\mathbb{K}$. Repeating this
process, we obtain a sequence $\left(  u_{k}\right)  _{k=1}^{\infty}$ in $V$
with $\left\Vert u_{k}\right\Vert =1$, for each $k\in\mathbb{N}$, such that
for each $N\geq1$ and any scalars $\left(  \lambda_{k}\right)  _{k=1}^{n+N}$,
\[
\left\Vert
%TCIMACRO{\tsum \limits_{k=1}^{n+N}}%
%BeginExpansion
{\textstyle\sum\limits_{k=1}^{n+N}}
%EndExpansion
\lambda_{k}u_{k}\right\Vert \geq\left(  1-\varepsilon_{N}\right)  \left\Vert
%TCIMACRO{\tsum \limits_{k=1}^{n+N-1}}%
%BeginExpansion
{\textstyle\sum\limits_{k=1}^{n+N-1}}
%EndExpansion
\lambda_{k}u_{k}\right\Vert \text{.}%
\]
Therefore,
\begin{align*}
\left\Vert
%TCIMACRO{\tsum \limits_{k=1}^{n+N-1}}%
%BeginExpansion
{\textstyle\sum\limits_{k=1}^{n+N-1}}
%EndExpansion
\lambda_{k}u_{k}\right\Vert  &  \leq%
%TCIMACRO{\tprod \limits_{k=1}^{m+1}}%
%BeginExpansion
{\textstyle\prod\limits_{k=1}^{m+1}}
%EndExpansion
\left(  1-\varepsilon_{N+k-1}\right)  ^{-1}\left\Vert
%TCIMACRO{\tsum \limits_{k=1}^{n+N+m}}%
%BeginExpansion
{\textstyle\sum\limits_{k=1}^{n+N+m}}
%EndExpansion
\lambda_{k}u_{k}\right\Vert \\
&  \leq2\left\Vert
%TCIMACRO{\tsum \limits_{k=1}^{n+N+m}}%
%BeginExpansion
{\textstyle\sum\limits_{k=1}^{n+N+m}}
%EndExpansion
\lambda_{k}u_{k}\right\Vert
\end{align*}
for each $m\in\mathbb{N}$. In particular,
\begin{equation}
\left\Vert
%TCIMACRO{\tsum \limits_{k=1}^{s}}%
%BeginExpansion
{\textstyle\sum\limits_{k=1}^{s}}
%EndExpansion
\lambda_{k}u_{k}\right\Vert \leq2\left\Vert
%TCIMACRO{\tsum \limits_{k=1}^{t}}%
%BeginExpansion
{\textstyle\sum\limits_{k=1}^{t}}
%EndExpansion
\lambda_{k}u_{k}\right\Vert \label{3.1}%
\end{equation}
for each $t\geq s\geq n$.

Now let us suppose $s\leq r\leq n$. Since the correspondence
\[%
%TCIMACRO{\tsum \limits_{k=1}^{n}}%
%BeginExpansion
{\textstyle\sum\limits_{k=1}^{n}}
%EndExpansion
\alpha_{k}u_{k}\mapsto%
%TCIMACRO{\tsum \limits_{k=1}^{n}}%
%BeginExpansion
{\textstyle\sum\limits_{k=1}^{n}}
%EndExpansion
\left\vert \alpha_{k}\right\vert
\]
defines a norm on $W_{1}$ and two norms are always equivalent in finite
dimensional spaces, there are positive constants $L$ and $M$, such that
\[
\left\Vert
%TCIMACRO{\tsum \limits_{k=1}^{n}}%
%BeginExpansion
{\textstyle\sum\limits_{k=1}^{n}}
%EndExpansion
\alpha_{k}u_{k}\right\Vert \leq L%
%TCIMACRO{\tsum \limits_{k=1}^{n}}%
%BeginExpansion
{\textstyle\sum\limits_{k=1}^{n}}
%EndExpansion
\left\vert \alpha_{k}\right\vert \leq M\left\Vert
%TCIMACRO{\tsum \limits_{k=1}^{n}}%
%BeginExpansion
{\textstyle\sum\limits_{k=1}^{n}}
%EndExpansion
\alpha_{k}u_{k}\right\Vert \text{.}%
\]
Hence,
\begin{equation}
\left\Vert
%TCIMACRO{\tsum \limits_{k=1}^{s}}%
%BeginExpansion
{\textstyle\sum\limits_{k=1}^{s}}
%EndExpansion
\lambda_{k}u_{k}\right\Vert \leq L%
%TCIMACRO{\tsum \limits_{k=1}^{s}}%
%BeginExpansion
{\textstyle\sum\limits_{k=1}^{s}}
%EndExpansion
\left\vert \lambda_{k}\right\vert \leq L%
%TCIMACRO{\tsum \limits_{k=1}^{r}}%
%BeginExpansion
{\textstyle\sum\limits_{k=1}^{r}}
%EndExpansion
\left\vert \lambda_{k}\right\vert \leq M\left\Vert
%TCIMACRO{\tsum \limits_{k=1}^{r}}%
%BeginExpansion
{\textstyle\sum\limits_{k=1}^{r}}
%EndExpansion
\lambda_{k}u_{k}\right\Vert \text{.} \label{3.2}%
\end{equation}
Combining the previous inequality with (\ref{3.1}) we have
\begin{equation}
\left\Vert
%TCIMACRO{\tsum \limits_{k=1}^{s}}%
%BeginExpansion
{\textstyle\sum\limits_{k=1}^{s}}
%EndExpansion
\lambda_{k}u_{k}\right\Vert \leq2M\left\Vert
%TCIMACRO{\tsum \limits_{k=1}^{t}}%
%BeginExpansion
{\textstyle\sum\limits_{k=1}^{t}}
%EndExpansion
\lambda_{k}u_{k}\right\Vert \text{,} \label{3.3}%
\end{equation}
for each $t\geq n$.

Finally, by (\ref{3.1}), (\ref{3.2}) and (\ref{3.3}), we conclude that, in
general, if $r,s\in\mathbb{N}$ are such that $s\leq r$, we have
\[
\left\Vert
%TCIMACRO{\tsum \limits_{k=1}^{s}}%
%BeginExpansion
{\textstyle\sum\limits_{k=1}^{s}}
%EndExpansion
\lambda_{k}u_{k}\right\Vert \leq C\left\Vert
%TCIMACRO{\tsum \limits_{k=1}^{r}}%
%BeginExpansion
{\textstyle\sum\limits_{k=1}^{r}}
%EndExpansion
\lambda_{k}u_{k}\right\Vert
\]
for a certain constant $C$. It is well known that it means that $\left(
u_{k}\right)  _{k=1}^{\infty}$ is a basic sequence.
\end{proof}

\begin{lemma}
\label{Lema3.5}\cite[Lemma 4.3]{Kalton} Let $W$ be a Banach space and $\left(
w_{k}\right)  _{k=1}^{\infty}$ be a regular basic sequence. Let $\left(
u_{k}\right)  _{n=1}^{\infty}$ be a sequence in $W$ such that $%
%TCIMACRO{\tsum \limits_{n=1}^{\infty}}%
%BeginExpansion
{\textstyle\sum\limits_{n=1}^{\infty}}
%EndExpansion
\left\Vert u_{k}\right\Vert <\infty$. If%
\[%
%TCIMACRO{\tsum \limits_{n=1}^{\infty}}%
%BeginExpansion
{\textstyle\sum\limits_{n=1}^{\infty}}
%EndExpansion
a_{k}\left(  w_{k}+u_{k}\right)  =0\Rightarrow a_{k}=0\text{,}%
\]
then $\left(  w_{k}+u_{k}\right)  _{n=1}^{\infty}$ is a basic sequence.
\end{lemma}

Now we are able to state the main result of this section.

\begin{theorem}
\label{Teo3.6}Let $\lambda\geq\mathfrak{c}$, $V$ be a $\lambda$-dimensional
Banach space and $W$ be a closed vector subspace of $V$. If $V\setminus W$ is
$\aleph_{0}$-lineable and $W$ has a regular basic sequence, then $V\setminus
W$ is $\left(  n,\mathfrak{c}\right)  $-spaceable, for every $n\in\mathbb{N}$.
\end{theorem}

\begin{proof}
Let $Z$ be an $n$-dimensional subspace of $V$ such that
\[
Z\setminus\left\{  0\right\}  \subset V\setminus W\text{.}%
\]
Let $\left\{  v_{1},\ldots,v_{n}\right\}  $ be a basis of $Z$. Since $W$ is
closed, the quotient space $V/W$ is Banach when endowed with the norm
$\left\Vert \cdot\right\Vert _{V/W}$ given by $\left\Vert \overline
{v}\right\Vert _{V/W}=\inf\left\{  \left\Vert v-w\right\Vert :w\in W\right\}
$, where $\overline{v}$ denotes the equivalence class of $v\in V$. Let
$\left(  y_{k}\right)  _{k=1}^{\infty}$ be a regular basic sequence in $W$.
For each $i=1,\ldots,n$, let%
\begin{equation}
x_{i}=\frac{\left(  v_{i}-y_{i}\right)  }{\left\Vert \overline{v}%
_{i}\right\Vert _{V/W}}\text{.} \label{3.4}%
\end{equation}
Thus $\overline{x}_{1},\ldots,\overline{x}_{n}$ are linearly independent. In
fact, if
\[%
%TCIMACRO{\tsum \limits_{i=1}^{n}}%
%BeginExpansion
{\textstyle\sum\limits_{i=1}^{n}}
%EndExpansion
\lambda_{i}\overline{x}_{i}=\overline{0}\text{,}%
\]
then
\[
\overline{0}=%
%TCIMACRO{\tsum \limits_{i=1}^{n}}%
%BeginExpansion
{\textstyle\sum\limits_{i=1}^{n}}
%EndExpansion
\frac{\lambda_{i}}{\left\Vert \overline{v}_{i}\right\Vert _{V/W}}\overline
{v}_{i}=\overline{%
%TCIMACRO{\tsum \limits_{i=1}^{n}}%
%BeginExpansion
{\textstyle\sum\limits_{i=1}^{n}}
%EndExpansion
\frac{\lambda_{i}}{\left\Vert \overline{v}_{i}\right\Vert _{V/W}}v_{i}%
}\text{.}%
\]
This implies that
\[%
%TCIMACRO{\tsum \limits_{i=1}^{n}}%
%BeginExpansion
{\textstyle\sum\limits_{i=1}^{n}}
%EndExpansion
\frac{\lambda_{i}}{\left\Vert \overline{v}_{i}\right\Vert _{V/W}}v_{i}\in
W\cap Z=\left\{  0\right\}  \text{.}%
\]
Since $\left\{  v_{1},\ldots,v_{n}\right\}  $ is a basis of $Z$, we conclude
that%
\[
\lambda_{i}=0\text{, for each }i=1,\ldots,n\text{.}%
\]
By Lemma \ref{Lema3.4}, we obtain a normalized basic sequence $\left(
\overline{x}_{k}\right)  _{k=1}^{\infty}$ on $V/W$, with $x_{1},\ldots,x_{n}$
as in (\ref{3.4}). Since $\left\Vert \overline{x}_{k}\right\Vert _{V/W}%
:=\inf\left\{  \Vert x_{k}-w\Vert:w\in W\right\}  $, for every $k>n$, there is
$w_{k}\in W$ such that
\[
\left\Vert x_{k}-w_{k}\right\Vert \leq\left\Vert \overline{x}_{k}\right\Vert
_{V/W}+2^{-k}=1+2^{-k}\text{.}%
\]
Let
\[
u_{k}:=\left\{
\begin{array}
[c]{ll}%
\left\Vert \overline{v}_{k}\right\Vert _{V/W}x_{k}\text{,} & \text{if }k\leq
n\text{,}\vspace{0.2cm}\\
2^{-k}\left(  x_{k}-w_{k}\right)  \text{,} & \text{if }k>n\text{,}%
\end{array}
\right.
\]
and let $\left(  a_{k}\right)  _{k=1}^{\infty}$ be a sequence in $\mathbb{K}$
such that
\begin{equation}%
%TCIMACRO{\tsum \limits_{k=1}^{\infty}}%
%BeginExpansion
{\textstyle\sum\limits_{k=1}^{\infty}}
%EndExpansion
a_{k}\left(  y_{k}+u_{k}\right)  =0\text{.} \label{3.5}%
\end{equation}
Then, in particular,%
\begin{equation}
\lim_{k\rightarrow\infty}\left\Vert a_{k}\right\Vert \left\Vert y_{k}%
+u_{k}\right\Vert =0\text{.} \label{3.6}%
\end{equation}
Since $\left(  y_{k}\right)  _{k=1}^{\infty}$ is a regular basic sequence,
there is $L>0$ such that $\left\Vert y_{k}\right\Vert \geq L$ for each
$k\in\mathbb{N}$. Thus, if $k>n$, then%
\begin{equation}
\left\Vert y_{k}+u_{k}\right\Vert \geq\left\Vert y_{k}\right\Vert -\left\Vert
u_{k}\right\Vert \geq L-2^{-k}\left(  1+2^{-k}\right)  \overset{k\rightarrow
\infty}{\longrightarrow}L>0\text{.} \label{3.7}%
\end{equation}
From (\ref{3.6}) and (\ref{3.7}) we conclude that
\begin{equation}
\lim_{k\rightarrow\infty}a_{k}=0\text{.} \label{3.8}%
\end{equation}

The inequality%
\begin{equation}%
%TCIMACRO{\tsum \limits_{k=n+1}^{\infty}}%
%BeginExpansion
{\textstyle\sum\limits_{k=n+1}^{\infty}}
%EndExpansion
\left\Vert u_{k}\right\Vert =%
%TCIMACRO{\tsum \limits_{k=n+1}^{\infty}}%
%BeginExpansion
{\textstyle\sum\limits_{k=n+1}^{\infty}}
%EndExpansion
\left\Vert 2^{-k}\left(  x_{k}-w_{k}\right)  \right\Vert \leq%
%TCIMACRO{\tsum \limits_{k=n+1}^{\infty}}%
%BeginExpansion
{\textstyle\sum\limits_{k=n+1}^{\infty}}
%EndExpansion
2^{-k}\left(  1+2^{-k}\right)  <\infty\text{,} \label{3.9}%
\end{equation}

combined with (\ref{3.8}), allow us to conclude that $%
%TCIMACRO{\tsum \limits_{k=1}^{\infty}}%
%BeginExpansion
{\textstyle\sum\limits_{k=1}^{\infty}}
%EndExpansion
a_{k}u_{k}$ converges absolutely and, hence, converges. Therefore, by
(\ref{3.5}), we have%
\[%
%TCIMACRO{\tsum \limits_{k=1}^{\infty}}%
%BeginExpansion
{\textstyle\sum\limits_{k=1}^{\infty}}
%EndExpansion
a_{k}y_{k}=-%
%TCIMACRO{\tsum \limits_{k=1}^{\infty}}%
%BeginExpansion
{\textstyle\sum\limits_{k=1}^{\infty}}
%EndExpansion
a_{k}u_{k}\text{.}%
\]
Consequently, $%
%TCIMACRO{\tsum \limits_{k=1}^{\infty}}%
%BeginExpansion
{\textstyle\sum\limits_{k=1}^{\infty}}
%EndExpansion
a_{k}u_{k}\in W$. Thus,%
\[
\overline{0}=%
%TCIMACRO{\tsum \limits_{k=1}^{\infty}}%
%BeginExpansion
{\textstyle\sum\limits_{k=1}^{\infty}}
%EndExpansion
a_{k}\overline{u}_{k}=%
%TCIMACRO{\tsum \limits_{k=1}^{n}}%
%BeginExpansion
{\textstyle\sum\limits_{k=1}^{n}}
%EndExpansion
a_{k}\overline{u}_{k}+%
%TCIMACRO{\tsum \limits_{k=n+1}^{\infty}}%
%BeginExpansion
{\textstyle\sum\limits_{k=n+1}^{\infty}}
%EndExpansion
a_{k}\overline{u}_{k}=%
%TCIMACRO{\tsum \limits_{k=1}^{n}}%
%BeginExpansion
{\textstyle\sum\limits_{k=1}^{n}}
%EndExpansion
a_{k}\left\Vert \overline{v}_{k}\right\Vert _{V/W}\overline{x}_{k}+%
%TCIMACRO{\tsum \limits_{k=n+1}^{\infty}}%
%BeginExpansion
{\textstyle\sum\limits_{k=n+1}^{\infty}}
%EndExpansion
a_{k}2^{-k}\overline{x}_{k}%
\]
and, since $\left\{  \overline{x}_{k}:k\in\mathbb{N}\right\}  $ is a basic
sequence in $V/W$, it follows that
\begin{equation}
a_{k}=0 \label{3.10}%
\end{equation}
for each $k\in\mathbb{N}$. By (\ref{3.9}) and (\ref{3.10}) we can invoke Lemma
\ref{Lema3.5} to conclude that the sequence $\left(  y_{k}+u_{k}\right)
_{k=1}^{\infty}$ is a basic sequence in $V$. Defining the (norm) closure of
$\operatorname*{span}\left\{  y_{k}+u_{k}:k\in\mathbb{N}\right\}  $ by $F$,
let us to prove that $F\setminus\left\{  0\right\}  \subset V\setminus W$. If
$v\in W\cap F$, then there are scalars $c_{k}$ such that
\[
v=%
%TCIMACRO{\tsum \limits_{k=1}^{\infty}}%
%BeginExpansion
{\textstyle\sum\limits_{k=1}^{\infty}}
%EndExpansion
c_{k}\left(  y_{k}+u_{k}\right)
\]
and
\[
\overline{0}=\overline{v}=%
%TCIMACRO{\tsum \limits_{k=1}^{\infty}}%
%BeginExpansion
{\textstyle\sum\limits_{k=1}^{\infty}}
%EndExpansion
\overline{c_{k}\left(  y_{k}+u_{k}\right)  }={%
%TCIMACRO{\tsum \limits_{k=1}^{\infty}}%
%BeginExpansion
{\textstyle\sum\limits_{k=1}^{\infty}}
%EndExpansion
}c_{k}\overline{u}_{k}={%
%TCIMACRO{\tsum \limits_{k=1}^{\infty}}%
%BeginExpansion
{\textstyle\sum\limits_{k=1}^{\infty}}
%EndExpansion
}d_{k}\overline{x}_{k}\text{,}%
\]
where%
\[
d_{k}=\left\{
\begin{array}
[c]{ll}%
c_{k}\left\Vert \overline{v}_{k}\right\Vert _{V/W}\text{,} & \text{if }k\leq
n\text{,}\vspace{0.2cm}\\
c_{k}2^{-k}\text{,} & \text{if }k>n\text{.}%
\end{array}
\right.
\]
Since $\left(  \overline{x}_{k}\right)  _{k=1}^{\infty}$ is a basic sequence
in $V/W$, it follows that $d_{k}=c_{k}=0$ for all $k\in\mathbb{N}$ and, hence
$v=0$. Thus $W\cap F=\left\{  0\right\}  $, that is,
\[
F\setminus\left\{  0\right\}  \subset V\setminus W\text{.}%
\]
Since
\[
y_{k}+u_{k}=y_{k}+\left\Vert \overline{v}_{k}\right\Vert _{V/W}x_{k}%
=y_{k}+\left\Vert \overline{v}_{k}\right\Vert _{V/W}\frac{v_{k}-y_{k}%
}{\left\Vert \overline{v}_{k}\right\Vert _{V/W}}=v_{k}\text{,}%
\]
for all $k=1,\ldots,n$, we have $Z\subset F$ and the result is done, since
$\dim F=\mathfrak{c}$.
\end{proof}

The next result is a consequence of Corollary \ref{Cor2.3} and Theorem
\ref{Teo3.6}.

\begin{corollary}
\label{Cor3.7}If $W$ is a closed proper subspace of a Banach space $V$ such
that $W$ has a regular basic sequence and $V\setminus W$ is $\aleph_{0}%
$-lineable, then%
\[
V\setminus W\text{ is }\left(  \alpha,\mathfrak{c}\right)  \text{-spaceable
if, and only if, }\alpha<\aleph_{0}\text{.}%
\]

\end{corollary}

\begin{proof}
Theorem \ref{Teo3.6} assures that $V\setminus W$ is $\left(  \alpha
,\mathfrak{c}\right)  $-spaceable, for every $\alpha<\aleph_{0}$. Conversely,
Corollary \ref{Cor2.3} guarantees that, if $\aleph_{0}\leq\alpha
\leq\mathfrak{c}$, then $V\setminus W$ is not $\left(  \alpha,\mathfrak{c}%
\right)  $-spaceable.
\end{proof}

Recall that a subspace $Y$ of a Banach space $X$ is
\textit{quasi-complemented} in $X$ if it is closed, and there exists a closed
linear subspace $Z$ of $X$ such that $Y\cap Z=\left\{  0\right\}  $ and $Y+Z$
is dense in $X$ (see \cite{Rosenthal, James, Lindenstrauss}). In
\cite{Lindenstrauss}, Lindenstrauss asks whether or not $c_{0}$ is
quasi-complemented in $\ell_{\infty}$, and, in \cite[Theorem 1.7]{Rosenthal},
Rosenthal shows that this is so. Furthermore, every separable subspace of
$\ell_{\infty}$ is quasi-complemented in $\ell_{\infty}$. In particular, if
$Y=c_{0}$ or $c$, the set $\ell_{\infty}\setminus Y$ is $\mathfrak{c}$-spaceable.

The next corollary characterizes the $\left(  \alpha,\mathfrak{c}\right)
$-spaceability of the sets $\ell_{\infty}\setminus c$ and $\ell_{\infty
}\backslash c_{0}$ giving more geometric contours to the Rosenthal's result:

\begin{corollary}
Both $\ell_{\infty}\setminus c$ and $\ell_{\infty}\setminus c_{0}$ are
$\left(  \alpha,\mathfrak{c}\right)  $-spaceable if, and only if,
$\alpha<\aleph_{0}$.
\end{corollary}

\begin{proof}
The canonical sequence $\left(  e_{n}\right)  _{n=1}^{\infty}\subset
\ell_{\infty}$ is a Schauder basis for $c_{0}$. Hence $\left(  e_{n}\right)
_{n=1}^{\infty}$ is a basic sequence for $c$. Since $\left\Vert e_{n}%
\right\Vert =1$ for all $n\in\mathbb{N}$, the sequence $\left(  e_{n}\right)
_{n=1}^{\infty}$ is a regular basic sequence for both $c_{0}$ and $c$.
Moreover, $c_{0}$ and $c$ are both closed in $\ell_{\infty}$. Hence, if we
take $V=\ell_{\infty}$ and $W=c_{0}$ or $c$ in Corollary \ref{Cor3.7}, we
conclude that both $\ell_{\infty}\setminus c$ and $\ell_{\infty}\setminus
c_{0}$ are $\left(  \alpha,\mathfrak{c}\right)  $-spaceable if, and only if,
$\alpha<\aleph_{0}$.
\end{proof}

The following result is a consequence of the proof of Theorem \ref{Teo3.6}.

\begin{theorem}
Let $\lambda\geq\mathfrak{c}$, $V$ be a $\lambda$-dimensional Banach space and
$W$ be a closed vector subspace of $V$. If $V\setminus W$ is $\mathfrak{c}%
$-lineable and $W$ has a regular basic sequence $\left(  y_{k}\right)
_{k=1}^{\infty}$, then $V\setminus W$ is pointwise $\mathfrak{c}$-spaceable.
\end{theorem}

\section{Dense lineability\label{Sec4}}

The following result was proved by Bernal et al. \cite{Bernal}:

\begin{theorem}
\label{Teo4.1}(\cite[Theorem 2.5]{Bernal}) Let $V$ be a metrizable separable
topological vector space and $W$ be a vector subspace of $V$. If $W$ has
infinite codimension, then $V\setminus W$ is $\aleph_{0}$-dense-lineable.
\end{theorem}

In this section, in some sense, we complement this result. We shall begin by
introducing the notion of $\left(  \alpha,\beta\right)  $-dense lineability:

\begin{definition}
Let $V$ be a $\lambda$-dimensional vector space endowed with a topology and
let\ $A$ be a non-empty subset of $V$. We say that $A$ is $\left(
\alpha,\beta\right)  $-dense lineable if it is $\alpha$\emph{-}lineable and
for each $\alpha$-dimensional subspace $W_{\alpha}\subset A\cup\left\{
0\right\}  $ there is a $\beta$-dimensional dense subspace $W_{\beta}$ of $V$
such that $W_{\alpha}\subset W_{\beta}\subset A\cup\left\{  0\right\}
$.\textrm{ }
\end{definition}

The main result of this section is a kind of variant of Theorem \ref{Teo4.1}:

\begin{theorem}
\label{Teo4.3}Let $V$ be a metrizable separable topological vector space and
let $W$ be a subspace of $V$. If $W$ has infinite codimension, then
$V\setminus W$ is $\left(  k,\aleph_{0}\right)  $-dense lineable, for each
$k\in\mathbb{N}$.
\end{theorem}

\begin{proof}
Fixing $k\in\mathbb{N}$, let $g_{1},\ldots,g_{k}$ be linearly independent
vectors such that
\[
G\setminus\left\{  0\right\}  \subset V\setminus W\text{,}%
\]
where $G:=\operatorname*{span}\left\{  g_{1},\ldots,g_{k}\right\}  $. Define
$M_{0}:=W\oplus G$. Let $\left\{  U_{n}:n\in\mathbb{N}\right\}  $ be a
countable open basis for the topology of $V$. Since $\operatorname{codim}%
W=\infty$ we have $M_{0}\subsetneq V$. This assures that $\operatorname{int}%
\left(  M_{0}\right)  =\varnothing$, and hence $V\setminus M_{0}$ is dense in
$V$. Let $g_{k+1}\in U_{1}\cap\left(  V\setminus M_{0}\right)  $ and define
$G_{1}=\operatorname*{span}\left\{  g_{1},\ldots,g_{k},g_{k+1}\right\}  $.
Since $\operatorname{codim}W=\infty$ we have $M_{1}:=W\oplus G_{1}\subsetneq
V$ and consequently $\operatorname{int}\left(  M_{1}\right)  =\varnothing$.
Therefore there is $g_{k+2}\in U_{2}\cap\left(  V\setminus M_{1}\right)  $.
Recursively, we obtain a countable set of vectors $\left\{  g_{k+1}%
,g_{k+2},\ldots,g_{k+n},\ldots\right\}  $ such that,
\[
g_{k+n+1}\in U_{n+1}\cap\left(  V\setminus M_{n}\right)  \text{,}%
\]
where
\[
M_{n}:=W\oplus G_{n}\text{ and }G_{n}:=\operatorname*{span}\left\{
g_{1},\ldots,g_{k},g_{k+1},\ldots,g_{k+n}\right\}  \text{.}%
\]
It is plain that $\left\{  g_{k+1},g_{k+2},\ldots,g_{k+n},\ldots\right\}  $ is
linearly independent and dense in $V$, consequently $\operatorname*{span}%
\left(  \left\{  g_{1},\ldots,g_{k+1},g_{k+2},\ldots,g_{k+n},\ldots\right\}
\right)  $ is $\aleph_{0}$-dimensional and dense in $V$.

If $v\in\operatorname*{span}\left(  \left\{  g_{1},\ldots,g_{k+1}%
,g_{k+2},\ldots,g_{k+n},\ldots\right\}  \right)  \setminus\left\{  0\right\}
$, then $v=\sum\limits_{i=1}^{k+N}a_{i}g_{i}$. Hence $v\in G_{N}$ and, since
$W\cap G_{N}=\left\{  0\right\}  $, we conclude that $v\notin W$. Therefore
\[
G\subset\operatorname*{span}\left(  \left\{  g_{1},\ldots,g_{k+1}%
,g_{k+2},\ldots,g_{k+n},\ldots\right\}  \right)  \subset\left(  V\setminus
W\right)  \cup\left\{  0\right\}  \text{.}%
\]

\end{proof}

It is well known (see \cite[Theorem 4.2]{Bernal}) that the set $L_{p}%
[0,1]\setminus\bigcup\nolimits_{q\in\left(  p,\infty\right)  }L_{q}[0,1]$ is
maximal dense-lineable, for each $p>0$. Considering $V=L_{p}[0,1]$ and
$W=\bigcup\nolimits_{q\in\left(  p,\infty\right)  }L_{q}[0,1]$ in Theorem
\ref{Teo4.3}, we have:

\begin{corollary}
Let $0<p<\infty$. The set $L_{p}[0,1]\setminus\bigcup\nolimits_{q\in\left(
p,\infty\right)  }L_{q}[0,1]$ is $\left(  n,\aleph_{0}\right)  $-dense
lineable, for each $n\in\mathbb{N}$.
\end{corollary}

In \cite[Section 2]{Nestoridis}, Nestoridis posed the following question:%
\[
\text{Does\ the\ set}\emph{\ }\ell_{\infty}\setminus c_{0}\emph{\ }%
\text{contain\ a\ dense\ linear\ subspace?}%
\]

In \cite[Theorem 1.2]{a2} Papathanasiou gave a positive answer to this
question. More precisely, he proved that the set $\ell_{\infty}\setminus
c_{0}$ is maximal dense-lineable. Nonetheless, $\ell_{\infty}\setminus c_{0}$
is not $\left(  n,\aleph_{0}\right)  $-dense lineable, for any $n\in
\mathbb{N}$, according to the next result:

\begin{theorem}
Let $F=c_{0}$ or $c$ and $n$ be a positive integer. The set $\ell_{\infty
}\setminus$ $F$ is $\left(  n,\aleph_{0}\right)  $-lineable but it is not
$\left(  n,\aleph_{0}\right)  $-dense lineable.
\end{theorem}

\begin{proof}
According to \cite[Proposition 1.5]{a2} the dimension of every dense linear
subspace of $\ell_{\infty}$ is $\mathfrak{c}$. This ensures that
\[
\ell_{\infty}\setminus F\text{ is not }\left(  n,\aleph_{0}\right)
\text{-dense lineable.}%
\]
On the other hand, it follows from Theorem \ref{Teo3.6} that
\[
\ell_{\infty}\setminus F\text{ is }\left(  n,\aleph_{0}\right)
\text{-lineable.}%
\]

\end{proof}

\section{Dense Pointwise Lineability\label{Sec5}}

In this section we introduce and explore the notion of pointwise
dense-lineability. This new concept is closely related to the notions of
pointwise lineability and dense-lineability (see \cite{AGSS, Aron2009}), but
it is of a stricter nature.

\begin{definition}
Let $\alpha\leq\lambda$ be cardinal numbers and let $V$ be a $\lambda
$-dimensional vector space endowed with a topology. Let\ $A$ be a non-void
subset of $V$. We say that $A$ is pointwise $\alpha$-dense-lineable if, for
each $x\in A$, there is an $\alpha$-dimensional dense subspace $W_{x}$ of $V$
such that
\[
x\in W_{x}\subset A\cup\left\{  0\right\}  \text{.}%
\]
When $\alpha=\lambda$, we say that $A$ is pointwise maximal dense-lineable.
\end{definition}

\begin{remark}
\label{Obs5.2}Let $V$ be an infinite dimensional metrizable separable
topological vector space and let $W$ be a nontrivial proper subspace of $V$.
In this case, for each $v\in V\setminus W$, we have $\mathbb{K}v\subset\left(
V\setminus W\right)  \cup\left\{  0\right\}  $. Hence, if $W$ has infinite
codimension, invoking Theorem \ref{Teo4.3} with $k=1$, we conclude that
$V\setminus W$ is pointwise $\aleph_{0}$-dense-lineable.
\end{remark}

In \cite[Theorem 2.5]{Bernal} the authors prove that if $V$ is a separable
metrizable topological vector space and $W$ is a subspace of $V$ with infinite
codimension, then $V\setminus W$ is $\aleph_{0}$-dense-lineable. The main
result of this section is a kind of complement of the aforementioned result:

\begin{theorem}
\label{Teo5.3}Let $V$ be an infinite dimensional separable metrizable
topological vector space and let $W$ be a subspace of $V$. If $W$ is dense in
$V$ and $V\setminus W$ is $\left(  m,\dim V\right)  $-lineable for some
$m\in\mathbb{N}$, then $V\setminus W$ is pointwise maximal dense-lineable.
\end{theorem}

\begin{proof}
Consider $v\in V\setminus W$ and let $m\in\mathbb{N}$ be such that $V\setminus
W$ is $\left(  m,\dim V\right)  $-lineable. It follows from Remark
\ref{Obs5.2} that there is an $\mathbb{\aleph}_{0}$-dimensional dense vector
subspace $M$ of $V$ such that $v\in M\subset\left(  V\setminus W\right)
\cup\left\{  0\right\}  $. Consider a vector subspace $E\subset M$ with
dimension $m$ and containing $v$. Since $V\setminus W$ is $\left(  m,\dim
V\right)  $-lineable, there is a vector subspace $Y\subset V$, with $\dim
Y=\dim V$ and such that%
\[
E\subset Y\subset\left(  V\setminus W\right)  \cup\left\{  0\right\}  \text{.}%
\]
Since $v\in Y$, we conclude that $V\setminus W$ is pointwise $\dim\left(
V\right)  $-lineable. Now we need to prove that $V\setminus W$ is
pointwise\emph{ }maximal dense-lineable. Fix $x_{0}\in W\setminus\left\{
0\right\}  $ and let us denote by $d$ a translation-invariant metric that
defines the topology of $V$. Since $V$ is separable and $W$ is dense in $V$,
we may choose a dense countable set $\left\{  p_{1},p_{2},\ldots\right\}  $ in
$V$\ and a countable set $\left\{  w_{n}:n\in\mathbb{N}\right\}  $ in $W$ such
that%
\begin{equation}
d\left(  p_{n},w_{n}\right)  \overset{n\rightarrow\infty}{\longrightarrow
}0\text{.} \label{5.1}%
\end{equation}
Note that (\ref{5.1}) yields the density of $\left\{  w_{n}:n\in
\mathbb{N}\right\}  $ in $V$. Define $x_{n}:=n^{-1}x_{0}$ and
\[
u_{n}:=\left\{
\begin{array}
[c]{ll}%
0\text{,} & \text{if }n=1\text{,}\vspace{0.2cm}\\
w_{n}+x_{n}\text{,} & \text{if }n>1\text{.}%
\end{array}
\right.
\]
Thus, if $n\geq2$, then%
\begin{equation}
d\left(  u_{n},w_{n}\right)  =d\left(  w_{n}+x_{n},w_{n}\right)  =d\left(
x_{n},0\right)  \overset{n\rightarrow\infty}{\longrightarrow}0\text{.}
\label{5.2}%
\end{equation}

Consider a basis $\left\{  y_{a}:a\in\Lambda\right\}  $ of $Y$. Since
$\Lambda$ is infinite, we may write $\Lambda=\bigcup_{n=1}^{\infty}\Lambda
_{n}$ (union of pairwise disjoint and infinite sets). Without loss of
generality we may assume that $v=y_{b}$, for some $b\in\Lambda_{1}$. Using the
continuity of the scalar multiplication once more, given $n\geq2$, for each
$a\in\Lambda_{n}$ there is $\lambda_{a}\in\mathbb{K}\setminus\left\{
0\right\}  $ such that $d\left(  \lambda_{a}y_{a},0\right)  \leq n^{-1}$. Let
us define%
\[
z_{a}:=\left\{
\begin{array}
[c]{ll}%
y_{a}\text{,} & \text{if }a\in\Lambda_{1}\text{,}\vspace{0.2cm}\\
\lambda_{a}y_{a}\text{,} & \text{if }a\in\Lambda\setminus\Lambda_{1}\text{.}%
\end{array}
\right.
\]
Hence, obviously, $\left\{  z_{a}:a\in\Lambda\right\}  $ is also a basis of
$Y$. For each $n\in\mathbb{N}$, let $\Gamma_{n}=\left\{  z_{a}+u_{n}%
:a\in\Lambda_{n}\right\}  $. Now, for each $n\in\mathbb{N}$, let us choose
$z_{a_{n}}+u_{n}\in\Gamma_{n}$. Observe that, if $n\geq2$, then, by
(\ref{5.1}) and (\ref{5.2}), we have%
\begin{align*}
d\left(  z_{a_{n}}+u_{n},p_{n}\right)   &  \leq d\left(  z_{a_{n}}+u_{n}%
,u_{n}\right)  +d\left(  u_{n},w_{n}\right)  +d\left(  w_{n},p_{n}\right) \\
&  =d\left(  z_{a_{n}},0\right)  +d\left(  u_{n},w_{n}\right)  +d\left(
w_{n},p_{n}\right)  \overset{n\rightarrow\infty}{\longrightarrow}0\text{.}%
\end{align*}
This means, in particular, that, $\Gamma=\bigcup_{n=1}^{\infty}\Gamma_{n}$ is
a dense subset of $V$ and so $\operatorname*{span}\left(  \Gamma\right)  $ is
dense in $V$.

Now we will show that $\Gamma$ is linearly independent. Suppose that there
exist $m\in\mathbb{N}$, $a_{1}\in\Lambda_{n_{1}},\ldots,a_{m}\in\Lambda
_{n_{m}}$ (with $a_{i}\neq a_{j}$ whenever $i\neq j$) and $\left(  \lambda
_{1},\ldots,\lambda_{m}\right)  \in\mathbb{K}^{m}\setminus\left\{  0\right\}
$ such that
\[%
%TCIMACRO{\tsum \limits_{i=1}^{m}}%
%BeginExpansion
{\textstyle\sum\limits_{i=1}^{m}}
%EndExpansion
\lambda_{i}\left(  z_{a_{i}}+u_{n_{i}}\right)  =0\text{.}%
\]
Observe that%
\[%
%TCIMACRO{\tsum \limits_{i=1}^{m}}%
%BeginExpansion
{\textstyle\sum\limits_{i=1}^{m}}
%EndExpansion
\lambda_{i}z_{a_{i}}\in Y\subset V\setminus W\text{ \ \ \ \ and \ \ \ \ }%
%TCIMACRO{\tsum \limits_{i=1}^{m}}%
%BeginExpansion
{\textstyle\sum\limits_{i=1}^{m}}
%EndExpansion
\lambda_{i}u_{n_{i}}\in W\text{.}%
\]
Since
\[%
%TCIMACRO{\tsum \limits_{i=1}^{m}}%
%BeginExpansion
{\textstyle\sum\limits_{i=1}^{m}}
%EndExpansion
\lambda_{i}z_{a_{i}}=-%
%TCIMACRO{\tsum \limits_{i=1}^{m}}%
%BeginExpansion
{\textstyle\sum\limits_{i=1}^{m}}
%EndExpansion
\lambda_{i}u_{n_{i}}%
\]
it follows that%
\[%
%TCIMACRO{\tsum \limits_{i=1}^{m}}%
%BeginExpansion
{\textstyle\sum\limits_{i=1}^{m}}
%EndExpansion
\lambda_{i}z_{a_{i}}\in Y\cap W=\left\{  0\right\}
\]
and so $\lambda_{1}=\cdots=\lambda_{m}=0$.

This same argument allows us to show that%
\[
\operatorname*{span}\left(  \Gamma\right)  \subset\left(  V\setminus W\right)
\cup\left\{  0\right\}  \text{.}%
\]
Since $v=z_{b}+u_{1}$, we conclude that $v\in\operatorname*{span}\left(
\Gamma\right)  $.
\end{proof}

\medskip

In \cite[Theorem 3.2]{FPT} the authors proved that the set $L_{p}\left[
0,1\right]  \setminus\bigcup\nolimits_{q\in\left(  p,\infty\right)  }%
L_{q}\left[  0,1\right]  $ is $\left(  1,\mathfrak{c}\right)  $-spaceable.
Considering $V=L_{p}[0,1]$ and $W=$ $\bigcup\nolimits_{q\in\left(
p,\infty\right)  }L_{q}\left[  0,1\right]  $ in Theorem \ref{Teo5.3}, since
$W$ is dense in $V$, we have:

\begin{corollary}
$L_{p}\left[  0,1\right]  \setminus\bigcup\nolimits_{q\in\left(
p,\infty\right)  }L_{q}\left[  0,1\right]  $ is pointwise maximal
dense-lineable, for every $p>0$.
\end{corollary}

\end{document}